\documentclass [leqno] {amsart}
  
\usepackage {amssymb}
\usepackage {amsthm}
\usepackage {amscd}

\usepackage[active]{srcltx}

\usepackage[all]{xy}
\usepackage {hyperref} 
\usepackage{mathrsfs}

\begin{document}

\bibliographystyle{alpha}
\theoremstyle{plain}
\newtheorem{proposition}[subsection]{Proposition}
\newtheorem{lemma}[subsection]{Lemma}
\newtheorem{corollary}[subsection]{Corollary}
\newtheorem{thm}[subsection]{Theorem}
\newtheorem{introthm}{Theorem}
\newtheorem*{thm*}{Theorem}
\newtheorem{conjecture}[subsection]{Conjecture}
\newtheorem{question}[subsection]{Question}
\newtheorem{fails}[subsection]{Fails}

\theoremstyle{definition}
\newtheorem{definition}[subsection]{Definition}
\newtheorem{notation}[subsection]{Notation}
\newtheorem{condition}[subsection]{Condition}
\newtheorem{example}[subsection]{Example}
\newtheorem{claim}[subsection]{Claim}

\theoremstyle{remark}
\newtheorem{remark}[subsection]{Remark}

\numberwithin{equation}{subsection}

%Matheumgebungen
\newcommand{\eq}[2]{\begin{equation}\label{#1}#2 \end{equation}}
\newcommand{\ml}[2]{\begin{multline}\label{#1}#2 \end{multline}}
\newcommand{\mlnl}[1]{\begin{multline*}#1 \end{multline*}}
\newcommand{\ga}[2]{\begin{gather}\label{#1}#2 \end{gather}}
\newcommand{\mat}[1]{\left(\begin{smallmatrix}#1\end{smallmatrix}\right)}

%xypic
\newcommand{\arir}{\ar@{^{(}->}}
\newcommand{\aril}{\ar@{_{(}->}}
\newcommand{\are}{\ar@{>>}}

% Pfeile xr , xl 
\newcommand{\xr}[1] {\xrightarrow{#1}}
\newcommand{\xl}[1] {\xleftarrow{#1}}
\newcommand{\lra}{\longrightarrow}
\newcommand{\inj}{\hookrightarrow}

% mathfrac, mathcal 
\newcommand{\mf}[1]{\mathfrak{#1}}
\newcommand{\mc}[1]{\mathcal{#1}}

% rm - Abkuerzungen 
\newcommand{\CH}{{\rm CH}}
\newcommand{\Gr}{{\rm Gr}}
\newcommand{\codim}{{\rm codim}}
\newcommand{\cd}{{\rm cd}}
\newcommand{\Spec} {{\rm Spec}}
\newcommand{\supp} {{\rm supp}}
\newcommand{\Hom} {{\rm Hom}}
\newcommand{\End} {{\rm End}}
\newcommand{\id}{{\rm id}}
\newcommand{\Aut}{{\rm Aut}}
\newcommand{\sHom}{{\rm \mathcal{H}om}}
\newcommand{\Tr}{{\rm Tr}}

% Abkuerzungen zu Bourbaki-Notation

% Zahlen
\renewcommand{\P} {\mathbb{P}}
\newcommand{\Z} {\mathbb{Z}}
\newcommand{\Q} {\mathbb{Q}}
\newcommand{\C} {\mathbb{C}}
\newcommand{\F} {\mathbb{F}}

%sonst
\newcommand{\OO}{\mathcal{O}}
\newcommand{\sO}{\mathcal{O}}
\newcommand{\sI}{\mathcal{I}}

\title[Vanishing of higher direct images]{Vanishing of the higher direct images of the structure sheaf}

\author{Andre Chatzistamatiou and Kay R\"ulling}
\address{Fachbereich Mathematik \\ Universit\"at Duisburg-Essen \\ 45117 Essen, Germany}
\email{a.chatzistamatiou@uni-due.de}
\address{Freie Universit\"at Berlin, Arnimallee 6, 14195 Berlin, Germany}
\email{kay.ruelling@fu-berlin.de}

\thanks{The first author has been supported by 
       the SFB/TR 45 ``Periods, moduli spaces and arithmetic of algebraic varieties'', the second author
        is supported by the ERC Advanced Grant 226257.}

\begin{abstract}
We prove that the higher direct images of the structure sheaf under a birational and projective morphism 
between excellent and regular schemes vanish.
\end{abstract}

\maketitle

%\tableofcontents

\section{Introduction}

In this article we prove the following theorem.

\begin{thm}\label{thm1} 
Let $f: X\to Y$ be a projective and birational morphism between excellent and regular schemes. 
Then the higher direct images of $\sO_X$ under $f$ vanish, that is,
\[R^i f_*\sO_X=0, \quad \text{for all $i\ge 1$.} \]
\end{thm}

In case $Y$ is of finite type over a characteristic zero field, this theorem was proved by Hironaka 
as a corollary of his work on the existence of resolutions of singularities, see \cite[(2), p.~144]{Hi}.
In a similar way, one can prove Theorem \ref{thm1} for $\dim Y=2$, see \cite[Prop.~1.2]{Li69}. 
If $Y$ is of finite type over a perfect field, then the theorem holds by \cite[Cor 3.2.10]{CR11}. The proof in {\it loc.~cit.}~uses the action of correspondences on Hodge cohomology. These methods do not seem to generalize to an arithmetic setup. Instead, in this article, we give a more direct proof, which relies on Grothendieck-Serre duality.

In view of \cite{Lipman-CM} and \cite{SanchoSalas}, we obtain the following application in commutative algebra.

\begin{thm}\label{thm-commutative-algebra}
	Let $R$ be an excellent regular local ring, and let $I\subset R$ be an ideal such that the blow up $X={\rm Proj}\, \bigoplus_{n\geq 0} I^n$ is regular. The following statements hold.
	\begin{enumerate}
		\item There is $e>0$ such that the Rees algebra $\bigoplus_{n\geq 0} I^{en}$ is Cohen-Macaulay.
		\item There is $e>0$ such that the associated graded algebra $\bigoplus_{n\geq 0} I^{en}/I^{e(n+1)}$ is Cohen-Macaulay.
	\end{enumerate}
\end{thm}
It is worth noting that assertion (1), or equivalently assertion (2), implies the vanishing $H^i(X,\OO_X)=0$, for all $i>0$, hence Theorem \ref{thm1} after an easy reduction to the local case.  

By using the main result of \cite{BBE07}, we obtain the following application, which was know if $X$ and $Y$ are defined over a finite field \cite[Theorem~1.1]{FR}.    
\begin{thm}\label{thm-points}
Let $f:X\to Y$ be as in Theorem~\ref{thm1}. Let $k$ be a finite field and let $s: \Spec\, k\to Y$ be a morphism. Denote by
$X_s= X\times_Y \Spec\, k$ the base change of $f$ along $s$. Then
the number of $k$-rational points of $X_s$ is congruent to 1 modulo the cardinality of $k$, that is,
\[|X_s(k)|\equiv 1 \text{ mod } |k|.\]
\end{thm}
See Section \ref{proofs} for a proof of this theorem. Theorem \ref{thm1} is a consequence of Theorem \ref{thm2} below. 
We have to introduce some notations to state it. 
From now on, all schemes in this introduction will be assumed to be separated, noetherian, 
excellent (see \cite[7.8]{EGAIV2}), and will admit a dualizing complex (see \cite[V]{Ha}). 
Let $f: X\to Y$ be a finite type morphism between integral schemes that is dominant, generically finite, and has a regular target $Y$. 
Then we define a morphism (see Proposition \ref{proposition-existence-fund-class})
\[c_f: \sO_X\to f^!\sO_Y\quad \text{in }D(Y).\]
It is a version of the fundamental class constructed for flat morphisms in \cite{ElZ78},
\cite{AEZ78} and \cite{ATJLL14}. In case $f$ is a proper complete intersection morphism of virtual dimension 0, 
the morphism $c_f$ corresponds by adjunction to the trace morphism $\tau_f: Rf_*\sO_X\to \sO_Y$ constructed 
in \cite[Thm 3.1]{BER}.

Furthermore, let $\sigma_A$ be a commutative square
\[\xymatrix{ V\ar[d]_f & A\ar[d]^{f_A}\ar[l]_{g_A}\ar@{}[dl]|-{\sigma_A}\\
                   Y & Z\ar[l]^g,
}\]
in which $f$ is a morphism of finite type and $g, g_A$ are proper. From duality theory, we obtain a natural transformation 
of functors $D^b_{\rm c}(Z)\to D^+_{\rm c}(V),$
\[\xi_{\sigma_A}: Rg_{A*}f_A^!\to f^!Rg_*.\]
\begin{thm}\label{thm2}
Consider the following  diagram 
$$
\xymatrix
{
	V \ar[d]_{f}
	&
	\\
	Y
	&
	Z,\ar[l]^{g}
}
$$
where we assume that the following conditions are satisfied:
\begin{enumerate}
	\item $V,Y,Z$ are integral schemes, $Y,Z$ are regular,
	\item $f$ is of finite type, dominant and generically finite, the base change $V\times_Y Z\xr{} Z$ is generically finite,
	\item $g$ is projective.
\end{enumerate}
Then the following equality holds in $\Hom_{D(V)}(\OO_V,f^!Rg_*\OO_Z)$:  
\begin{multline*}
	[\OO_V\xr{c_f} f^!\OO_Y \xr{f^!(g^*)} f^!Rg_*\OO_Z]= \\ 
          \sum_{A} \ell_A \cdot [\OO_V\xr{g_A^*} R g_{A*}\OO_A  \xr{Rg_{A*}(c_{f_A})} 
            R g_{A*} f^!_A\OO_Z \xr{\xi_{\sigma_A}}  f^!Rg_*\OO_Z ],
\end{multline*}
where the sum runs over all irreducible components $A$ of $V\times_Y Z$ that {\em dominate} $Z$,
$\ell_A$ is the multiplicity of $A$ in the generic fiber over $Z$, and $\sigma_A$ is a commutative diagram
as above, where $g_A$ and $f_A$ are induced by the composition of the closed immersion $A\inj V\times_Y Z$ followed 
by the projection to $V$ and $Z$, respectively.
\end{thm}
The formulation of this theorem is reminiscent of intersection theory. Indeed, methods from intersection theory are used
implicitly in the proof, which occupies most of the paper. 
In Section \ref{fc} we define the map $c_f$ and establish its main properties; in Section \ref{xi}
we give the definition and main properties for the map $\xi_\sigma$; in Section \ref{reductions}
we give the main reduction steps for the proof of Theorem \ref{thm2}; in Section \ref{proofs}
we  prove the above theorems.

Let us conclude with a list of open questions to which we hope to come back in the future:
\begin{enumerate}
\item Is Theorem \ref{thm1} or Theorem \ref{thm2} also true if one replaces "projective" by "proper"?
\item Let $S$ be an excellent scheme  and $f:X\to S$ and $g:Y\to S$ regular $S$-schemes. Assume 
        $X$ and $Y$ are properly birational over $S$, i.e. there exist birational and proper $S$-morphisms
          $V\to X$, $V\to Y$. Do we have an isomorphism $Rf_*\sO_X\cong Rg_*\sO_Y$?
         (In case $S$ is separated and of finite type over a perfect field this holds by \cite[Thm 1]{CR11}.)
\item What kind of singularities can we allow for $Y$ in order that the vanishing of Theorem \ref{thm1} still holds?
\end{enumerate}

\subsection*{Acknowledgements} We are grateful to  H{\'e}l{\`e}ne Esnault for her constant support and helpful remarks. We also thank the anonymous referee for pointing out that we obtain Theorem \ref{thm-commutative-algebra} as an application.

\subsection*{Conventions}\label{convention}
All schemes in the Sections \ref{fc} - \ref{reductions}
are assumed to be separated, noetherian and to admit a dualizing complex. 
We say a morphism $f: X\to Y$ is projective if it can be factored as a closed immersion $X\inj \P^n_Y$ 
followed by the projection $\P^n_Y\to Y$.
For a scheme $X$, we use the notation $D^*(X)$ and  $D^*_{{\rm c}}(X)$, with $*\in \{-,+,b\}$, as in \cite{Ha}, \cite{Co}.

\section{Fundamental class}\label{fc}

Let $f:X\xr{} Y$ be a map of finite type. By \cite[\textsection3.3]{Co} we have 
$$
f^!:D^+_{{\rm c}}(Y) \xr{} D^+_{{\rm c}}(X)
$$
at our disposal. Whenever $f$ is proper, the trace ${\rm Tr}_f:Rf_*\circ f^!\xr{} 1$ induces an isomorphism 
\begin{equation}\label{equation-adj}
\Hom_{D(X)}(F,f^!G) \cong \Hom_{D(Y)}(Rf_*F,G),
\end{equation}
provided that $F\in D^-_{{\rm c}}(X)$ and $G\in D^+_{{\rm c}}(Y)$. In particular, we obtain a natural transformation ${\rm Tr}^{\vee}_f:1\xr{} f^!\circ Rf_*$ of functors $D^b_{{\rm c}}(X)\xr{} D^+_{{\rm c}}(X)$ satisfying 
\begin{align}
	\id_{Rf_*(F)} &= \Tr_f(Rf_*F)\circ Rf_*(\Tr_f^\vee(F)), \label{adj1}\\
	\id_{f^!(G)}  &=  f^!\Tr_f(G)\circ \Tr_f^\vee(f^!G), \label{adj2}      
\end{align}
for  $F\in D^b_{\rm c}(X)$ and  $G\in D^b_{\rm c}(Y)$ such that $f^!G\in  D^b_{\rm c}(X)$. 

For a closed immersion $f$ the functor $f^!$ is right adjoint to $Rf_*$ when considered as functors on $D^+_{{\rm c}}$.

\subsection{Fundamental class}

Recall from \cite[VIII]{SGA6} that a morphism $f: X\to Y$ is a {\em complete intersection of virtual relative
dimension 0 (ci0 for short)} if any point $x\in X$ has an open neighborhood $U\subset X$, 
such that the restriction of $f$ to $U$  factors as 
\begin{equation} \label{ci0}
	f_{|U}=\pi\circ i,
\end{equation}
where $i:U\inj P$ is a regular closed immersion of codimension $\codim_P(U)=:n$  and $\pi:P\to Y$ is a smooth morphism of 
relative dimension $\codim_P(U)$. We collect the following facts from \cite{BER}.

Let $f:X\to Y$ be a ci0 morphism. There exists an invertible sheaf on $X$, called  {\em the canonical dualizing sheaf}
and denoted by $\omega_{X/Y}$, which on any open subset $U\subset X$ on which $f$ factors as in \eqref{ci0}
is isomorphic to
\begin{equation}\label{equation-dualizing-sheaf1}
	(\omega_{X/Y})_{|U}\cong \bigwedge^n (\sI/\sI^2)^\vee\otimes_{\sO_X} i^*\Omega^n_{P/Y},
\end{equation}
where $\sI$ is the ideal sheaf of the regular closed  immersion $U\inj P$ \cite[A.2]{BER}.
If $g: Z\to X$ is another ci0 morphism, then there is a canonical isomorphism (see \cite[A.5]{BER})
$$
	\zeta'_{g,f}: \omega_{Z/Y}\xr{\simeq} \omega_{Z/X}\otimes_{\sO_Z}g^*\omega_{X/Y}.
$$
There is a canonical isomorphism 
$$
	\lambda_f: \omega_{X/Y}\xr{\simeq} f^!\sO_Y
$$
in $D(X)$ (see \cite[B1]{BER}), and a global section 
$$
	\delta_f\in H^0(X,\omega_{X/Y})
$$
with the following property. For any open subset $U\subset X$ on which $f$ factors as in \eqref{ci0} and such that the ideal sheaf $\sI$ of $i: U\inj P$ is generated by a regular sequence $t_1,\ldots, t_n$,
$\delta_f$ is mapped to the following element under the isomorphism \eqref{equation-dualizing-sheaf1}:
\begin{equation}\label{delta2}
	\delta_{f|U}= (\bar{t}_1^\vee\wedge\ldots \wedge\bar{t}_n^\vee )\otimes i^*(dt_n\wedge\ldots\wedge dt_1),
\end{equation}
where $(\bar{t}_j^\vee)_j$ is the dual of the basis  $(i^*(t_j))_j$ of $\sI/\sI^2$ \cite[A.7]{BER}.
(If $\sI$ is the zero ideal, this means that $\delta_{f|U}=1\in \sO_U$.)

\begin{definition}\label{defn-FundamentalClass-ci0}
Let $f: X\to Y$ be a ci0 morphism. Then we define the morphism 
\[c_f: \sO_X\to f^!\sO_Y \]
to be the composition $\lambda_f \circ \delta_f$.
\end{definition}

\begin{proposition}\label{prop-properties-FundCl-ci0}
Assume $f: X\to Y$ is a ci0 morphism.
\begin{enumerate}
\item Let $u: U\to X$ be an \'etale morphism. Then $f\circ u: U\to Y$ is ci0 and 
\[c_{f\circ u}= u^*c_f,\]
where $u^*$ is the morphism 
\[u^*: H^0(X, f^!\sO_Y)\to H^0(X, Ru_*u^*f^!\sO_Y)= H^0(U, (f\circ u)^!\sO_Y).\] 
\item Assume $g: Z\to X$ is another ci0 morphism. Then $c_{f\circ g}$ is equal to the composition
         \[\sO_Z\xr{c_g} g^!\sO_X\xr{g^!(c_f)} g^!f^!\sO_Y = (f\circ g)^!\sO_Y.\]
\item Assume $f$ is finite and flat. Then the composition 
\[f_*\sO_X\xr{f_*(c_f)} f_*f^!\sO_Y\xr{\Tr_f} \sO_Y\]
equals the classical trace morphism $\text{trace}_{X/Y}: f_*\sO_X\to \sO_Y$.
\item Suppose $X$ is integral. If $f$ is not dominant then $c_f=0$.
\end{enumerate}
\end{proposition}

\begin{proof}
(2) follows from \cite[Prop.~A.8, (i) and Prop.~B.2]{BER}. Now  (1) follows from (2) and the fact that 
the composition
\[\sO_U\xr{c_u} u^!\sO_X\cong u^*\sO_X\cong \sO_U\]
is equal to the identity, which follows directly from \eqref{delta2}, taking the factorization $i=\id$ and $\pi=u$. 
The composition in (3) equals by its very definition the
morphism $\tau_f$ constructed in \cite[(B.7.3)]{BER} and hence the statement of (3) follows from
\cite[Thm.~3.1, (iii)]{BER}.

Finally for (4), we can assume that $f$ factors as $f=\pi\circ i$ with $i:X\inj P$ a regular closed immersion of codimension
$n$ and $\pi: P\to Y$ a smooth morphism of relative dimension $n$. Let $x\in P$ be the generic point of $X$, and
set $y:=\pi(x)\in Y$. Since $X$ is regular at $x$ so is $P$, and thus $Y$ is regular at $y$.
By assumption,  $c:=\dim \sO_{Y,y}\ge 1$. Let $t_1,\ldots, t_c$ be a regular system of parameters,
which generates the maximal ideal in $\sO_{Y,y}$. Further 
$\sO_{P,x}\otimes_{\sO_{Y,y}}\kappa(y)=\sO_{\pi^{-1}(y),x}$ is a regular local ring of dimension  
$n-c$, hence we find elements $s_{c+1},\ldots, s_n$ in the maximal ideal $\frak{m}_x$ of $\sO_{P,x}$
lifting a regular sequence of parameters of $\sO_{P,x}\otimes_{\sO_{Y,y}}\kappa(y)$.
We see that the sequence 
\[\pi^*(t_1),\ldots,\pi^*(t_c), s_{c+1},\ldots s_n\]
generates $\frak{m}_x$
and is thus a regular sequence of parameters for the regular and $n$-dimensional ring $\sO_{P,x}$.
Therefore after shrinking $X$ and $Y$ further we may assume that the ideal sheaf $\sI$ of $i$ is
generated by a sequence as above. Since $d\pi^*(t_1)=0$ in $i^*\Omega^1_{P/Y}$,
it follows immediately from the description of $\delta_f$ in \eqref{delta2} that $c_f$ vanishes.
\end{proof}

\begin{proposition}\label{proposition-vanishing-coh-of-support-dualizing-complex}
Let $Y$ be a regular scheme and  $f:X\xr{} Y$  a morphism of finite type with the following property: 
\eq{*}{ \dim\sO_{Y, f(\eta)}={\rm trdeg}(k(\eta)/k(f(\eta))),\quad \text{for any generic point } \eta\in X.}
For every closed subscheme $Z\subset X$ of codimension $\geq c$ we have
	$$
	\mathcal{H}^i(R\underline{\Gamma}_Zf^!\OO_Y)=0, \quad \text{for all $i<c$.}
	$$
 In particular,  $f^!\sO_Y=\tau_{\ge 0} f^!\sO_Y$ and the restriction morphism
$\Hom(\OO_X,f^!\OO_Y)\xr{} \Hom(\OO_U,f_{\mid U}^!\OO_Y)$ is injective 
  for all dense open subsets $U\subset X$, and is an isomorphism if $\codim_X(X\backslash U)\geq 2$.
\end{proposition}
\begin{proof}
Since $f^!\OO_Y$ is a dualizing complex, it is a CM complex for a shifted codimension filtration (see \cite[IV, \S3]{Ha}). 
By assumption \eqref{*}, \cite[(3.2.4), p.~129]{Co} and \cite[(3.3.36), p.~145]{Co} this shift is 0.
	\end{proof}

\begin{remark}\label{rmk-cond*}
\begin{enumerate}
\item Let $f: X\to Y$ be a morphism of finite type between irreducible schemes, which is dominant
and generically finite. Then $f$ satisfies condition \eqref{*}.
\item If $f: X\to Y$ is a finite-type morphism between regular schemes that satisfies condition \eqref{*},
then it is a ci0 morphism. (It is clear that $f$ is a complete intersection morphism and it follows from 
\cite[Prop. (5.6.4)]{EGAIV2} that its virtual relative dimension is 0.) 
\end{enumerate}
\end{remark}

\begin{proposition}\label{proposition-existence-fund-class}
	Let $f:X\xr{} Y$ be a finite type morphism between integral and excellent schemes satisfying
	condition \eqref{*}. Assume that $Y$ is regular. There is a unique morphism in $D(X)$, 
	$$
	c_f:\OO_X\xr{} f^!\OO_Y,
	$$
	such that the restriction to the open subset of regular points $X_{{\rm reg}}$ is the class from 
     Definition \ref{defn-FundamentalClass-ci0} for the ci0 morphism $f_{|X_{\rm reg}}$.
   Furthermore, $c_f$ satisfies the analog of  the properties (1)-(4) of 
     Proposition \ref{prop-properties-FundCl-ci0}.
        
	\begin{proof}
		Uniqueness follows from Proposition \ref{proposition-vanishing-coh-of-support-dualizing-complex}. 
For the construction, let $\nu:\tilde{X}\xr{} X$ be the normalization. We may construct $c_{f\circ \nu}$ 
by restricting to the regular locus and apply Proposition \ref{proposition-vanishing-coh-of-support-dualizing-complex}. 
We set 
\eq{proposition-existence-fund-class1}{
		c_f=\left[\OO_X\xr{\nu^*} \nu_*\OO_{\tilde{X}} \xr{\nu_*(c_{f\circ \nu})} \nu_*(f\circ \nu)^!\OO_Y 
       = \nu_*\nu^! f^!\OO_Y \xr{{\rm Tr}_{\nu}} f^!\OO_Y\right]. }
The second statement follows from Proposition \ref{proposition-vanishing-coh-of-support-dualizing-complex}.
\end{proof}
\end{proposition}

\begin{lemma}\label{lemma-push-forward}
	Assumptions as in Proposition \ref{proposition-existence-fund-class}. Let $g:Z\xr{} X$ be a proper morphism 
between integral schemes. Assume that $f\circ g: Z\to Y$ satisfies condition \eqref{*}. 
     	Then 
	$$
	\deg(Z/X)\cdot c_f = \left[\OO_X\xr{g^*} Rg_*\OO_Z\xr{Rg_*(c_{f\circ g})} Rg_*g^!f^!\OO_Y 
                                                         \xr{{\rm Tr}_g} f^!\OO_Y \right].
	$$
	\begin{proof}
We may assume that $f$ and $g$ are dominant and hence generically finite, because both sides vanish otherwise.
In view of Proposition \ref{proposition-vanishing-coh-of-support-dualizing-complex} it suffices to prove the assertion
after restricting to a dense open subset of $X$. We can therefore assume that $X, Z$ are regular and $g$ is finite and flat. Then the statement follows from Proposition \ref{prop-properties-FundCl-ci0}, (2),  (3). 
	\end{proof}
\end{lemma}

\section{The twisted base change map}\label{xi}
\begin{definition}\label{defn-ad-square}
 Let $\sigma$ be a commutative diagram 
\begin{equation}\label{equation-cartesian-diagram}
\xymatrix{
	X \ar[d]_{f} 
	&
	A \ar[l]_{g_1}\ar[d]^{f_1}\ar@{}[dl]|-{\sigma}
	\\
	Y
	&
	Z\ar[l]^{g}	
}
\end{equation}
of finite type morphisms. We say $\sigma$ is an {\em admissible square} if $g$ and $g_1$ are proper. 
We define the natural transformation of functors $D^b_{\rm c}(Z)\to D^+_{\rm c}(X)$
\[\xi_\sigma: Rg_{1*} f_1^!\to f^!Rg_*\]
to be the composition
\[\xi_\sigma: R g_{1*} f_1^!\xr{{\rm Tr}^{\vee}_g} R g_{1*}f^!_1 g^!Rg_* \xr{=} R g_{1*} g^!_1f^!R g_* 
                  \xr{{\rm Tr}_{g_1}} f^!Rg_*.\]
\end{definition}

\begin{lemma}\label{lem-xi}
Let $\sigma$ be an admissible square as in \eqref{equation-cartesian-diagram}.
\begin{enumerate}
\item Let 
\[\xymatrix{ X\ar[d]_f & A\ar[l]_{g_1}\ar[d]|-{f_1}\ar@{}[dl]|-{\sigma} & 
                                                                               B\ar[l]_{h_1}\ar[d]^{f_2}\ar@{} [dl]|-{\sigma_1}\\
                    Y            & Z \ar[l]^g & V\ar[l]^h }\]
be two admissible squares and denote by $\sigma_2$ their composition. We have 
\[\xi_{\sigma_2}(F)= \xi_\sigma(Rh_*(F))\circ Rg_{1*}(\xi_{\sigma_1}(F)),\quad F\in D^b_{\rm c}(V).\]
\item Let 
\[\xymatrix{ W\ar[d]_e & B\ar[l]_{g_2}\ar[d]^{e_1}\ar@{}[dl]|-{\sigma_1}\\
                    X\ar[d]_f & A\ar[l]|-{g_1}\ar[d]^{f_1}\ar@{}[dl]|-{\sigma}\\
                    Y            & Z\ar[l]^g       }\]
    be two admissible squares and denote by $\sigma_2$ their composition. If $G\in D^b_{c}(Z)$ and $f_1^!G, f_1^!g^!Rg_*G\in D^b(A)$, then we have 
\[\xi_{\sigma_2}(G)= e^!(\xi_{\sigma}(G))\circ \xi_{\sigma_1}(f_1^!(G)).\]
\item  Assume the morphisms $f, f_1$ are proper. If $G\in D^b_{c}(Z)$ and $f_1^!G\in D^b(A)$, then $\xi_\sigma$ equals the composition
	\begin{equation}\label{equation-with-f}
		R g_{1*}f_1^!(G)\xr{{\rm Tr}^{\vee}_f} f^!Rf_*R g_{1*} f_1^!(G) \xr{=} f^!Rg_*R f_{1*}f_1^!(G)  \xr{f^!Rg_*({\rm Tr}_{f_1})} f^!Rg_*(G).
	\end{equation}
\item Assume $\sigma$ is cartesian and $f$ is \'etale. Then $\xi_\sigma$ equals the base change isomorphism $R g_{1*}f_1^*\xr{\simeq} f^* R g_*$.
\item Assume the morphism $f, f_1$ are closed immersions. Then 
     $f_*(\xi_\sigma)$ is the morphism
    \[Rg_* R\sHom_Z(f_{1*}\sO_A, -)\to Rg_*R\sHom(Lg^*f_*\sO_X, -)\cong R\sHom_Y(f_*\sO_X, Rg_*(-)),\]
    where the first map is induced by the natural map  $Lg^*f_*\sO_X\to f_{1*}\sO_A$.
\item Assume that the square $\sigma$ is tor-independent. If $G\in D^b_c(Z)$ satisfies $g^!Rg_*(G)\in D^b(Z)$ then $\xi_\sigma(G)$ is an isomorphism.
\item Suppose $j:U\xr{} Y$ is an open immersion. Denote by $\sigma_U$ the commutative square 
$$
\xymatrix{
 	f^{-1}(U) \ar[d]_{\bar{f}} 
	&
	(f\circ g_1)^{-1}(U) \ar[l]_{\bar{g}_1}\ar[d]^{\bar{f}_1}\ar@{}[dl]|-{\sigma_U}
	\\
	U
	&
	g^{-1}(U) \ar[l]^{\bar{g}}	
}
$$
induced by $\sigma$.
Then 
$
\xi_{\sigma_U}=j^*\xi_{\sigma},
$
where we use the natural transformations $j^*Rg_{1*}f_1^!=R\bar{g}_{1*}\bar{f}_1^!j^*$ and $j^*f^!Rg_{*}=\bar{f}^!R\bar{g}_{*}j^*$.
\end{enumerate}
\end{lemma}
\begin{proof}
Claim (1) follows directly from \cite[Lem 3.4.3 (TRA1)]{Co}. Claim (2) follows from the definition and \eqref{adj2} by a straightforward computation. 

Assume $f$ and $f_1$ are proper. 
In order to prove Claim (3), note that \eqref{equation-with-f} is adjoint to $Rg_*({\rm Tr}_{f_1}(G))$. It follows easily from \cite[Lem 3.4.3 (TRA1)]{Co} and \eqref{adj1} that $\xi_{\sigma}(G)$ has the same property.  

Claim (4) and (7) follow from \cite[Lem 3.4.3 (TRA4)]{Co} and \eqref{adj1}.
Claim (5) follows in the same way as (3) by using the fact that $\Tr_{f_1}$ on the left-hand side of the composition in (5)
is given by precomposition with $\sO_Z\to f_{1*}\sO_A$ and $\Tr_f$ on the right-hand side is given
by precomposition with $\sO_Y\to f_*\sO_X$.

Let us prove (6). The question is local on $X$, we can thus assume that $f$ factors as 
$X\xr{i}U\xr{j}P:=\P^n_Y\xr{\pi}Y$, where $i$ is a closed immersion, $j$ is an open immersion, and
$\pi$ is the projection. We can factor $\sigma$ into three admissible cartesian squares:
\[\xymatrix{X\ar[d]_{i}  & A\ar[d]^{i_A}\ar[l]_{g_1}\ar@{}[dl]|-{\sigma_i}\\
                 U\ar[d]_{j}  & U_Z\ar[d]^{j_A}\ar[l]|-{g_U}\ar@{}[dl]|-{\sigma_j}\\
                  P\ar[d]_{\pi} & P_Z\ar[l]|-{\bar{g}_1}\ar[d]^{\pi_1}\ar@{}[dl]|-{\bar{\sigma}}\\
                Y & Z\ar[l]^{g}.
}\]
The square $\sigma_i$ is tor-independent hence $\xi_{\sigma_i}$ is an isomorphism by (5). It follows from (4) that $\xi_{\sigma_j}$ is an isomorphism. Since $\pi_1$ and $\pi_1\circ j_A$ are smooth (hence bounded complexes are mapped to bounded complexes via $(\cdot)^!$), we may use (2) and prove that $\xi_{\bar{\sigma}}(G)$ is an isomorphism.
Set $\omega_{\pi}:=\Omega^n_{P/Y}$, we define the isomorphism $\xi_{\bar{\sigma}}': R\bar{g}_{1*} \pi_1^!\to \pi^!R g_*$ as the composition
\mlnl{
R\bar{g}_{1*} \pi_1^! \cong R\bar{g}_{1*}(\omega_{\pi_1}[n]\otimes \pi_1^*(-) ) 
                                    \cong R\bar{g}_{1*}(\bar{g}_1^*\omega_{\pi}[n]\otimes \pi_1^*(-) )\\
                                      \cong \omega_{\pi}[n]\otimes R\bar{g}_{1*}\pi_1^*
                                      \cong \omega_{\pi}[n]\otimes \pi^*R g_* \cong \pi^!R g_*.
}
It suffices to show $\xi_{\bar{\sigma}}'=\xi_{\bar{\sigma}}$, which by (3) is equivalent to 
\[\Tr_{\pi}(Rg_*G)\circ R\pi_*(\xi_{\bar{\sigma}}'(G))= Rg_*(\Tr_{\pi_1}(G)),\quad G\in D^b_c(Z).\]
This follows directly from the definition of the projective trace (see \cite[2.3]{Co}) and \cite[(2.4.1)]{Co}.
\end{proof}

\section{Reductions}\label{reductions}
\subsection{Setup}
We consider the diagram 
$$
\xymatrix
{
	V \ar[d]_f
	&
	\\
	Y
	&
	Z,\ar[l]^{g}
}
$$
from Theorem \ref{thm2}. Recall that:
\begin{enumerate}
	\item $V,Y,Z$ are integral noetherian excellent schemes, $Y,Z$ are regular.
	\item $f$ is of finite type, dominant and generically finite 
              and the base change $V\times_Y Z\xr{} Z$ is generically finite.
	\item  $g$ is projective.
\end{enumerate}
For an irreducible component $A$ of $V\times_Y Z$, we denote by $\ell_A$ the multiplicity of $A$ in the generic fiber over
$Z$; if $A$ does not dominate $Z$ then $\ell_A=0$. We denote by $g_A$ and $f_A$ the composition of the
closed immersion $A\inj V\times_Y Z$ with the projection to $V$ and $Z$, respectively, and
by $\sigma_A$ the corresponding admissible square. 

\begin{notation}\label{not-equation}
With the above notations and assumptions, we say {\em $E(V\xr{} Y\xl{} Z)$ holds} if the following
equality holds in $\Hom_{D(V)}(\OO_V,f^!Rg_*\OO_Z)$:  
\begin{multline*}
	[\OO_V\xr{c_f} f^!\OO_Y \xr{f^!(g^*)} f^!Rg_*\OO_Z]= \\ 
  \sum_{A} \ell_A \cdot [\OO_V\xr{g_A^*} R g_{A*}\OO_A  \xr{R g_{A*}(c_{f_A})} R g_{A*}f^!_A\OO_Z 
       \xr{\xi_{\sigma_A}}  f^!Rg_*\OO_Z ],
\end{multline*}
where the sum runs over all irreducible components of $A$ with $\ell_A\neq 0$. 
(Notice that by condition (1) and (2) above $c_f$ and $c_{f_A}$ are defined.)
\end{notation}

\begin{lemma}\label{lem-finite-flat-case}
If $f$ is finite and flat, then $E(V\xr{f} Y\xl{g} Z)$ holds. 
\end{lemma}
\begin{proof}
The cartesian square 
$$
\xymatrix
{
	V \ar[d]_f
	&
	V\times_{Y} Z \ar[l]_{g_1} \ar[d]^{f_1}
	\\
	Y
	&
	Z,\ar[l]^{g}
}
$$
is tor-independent, and $g^!$ preserves bounded complexes (since $Y,Z$ are regular). Therefore we can use Lemma \ref{lem-xi}(6) and Proposition \ref{proposition-vanishing-coh-of-support-dualizing-complex} to obtain an injective map
$$
\Hom(\OO_V,f^!Rg_*\OO_Z) = \Gamma(V,g_{1*}\mathcal{H}^0(f_1^{!}\OO_Z)) \inj{} \Gamma(f^{-1}(U)\times_U g^{-1}(U),g_{1*}\mathcal{H}^0(f_1^{!}\OO_Z))  
$$
for every open $U\subset Y$ such that $g^{-1}(U)\neq \emptyset$. Therefore we may replace $Y$ by any such $U$. In particular, we may suppose that, for every irreducible component $A$ of $V\times_Y Z$, the induced map $f_A$ is flat. 

By  Proposition \ref{proposition-existence-fund-class}, (3), the maps $c_{f}$ and $c_{f_A}$ are adjoint to 
the trace maps.  By Lemma \ref{lem-xi}, (3) we have to show 
		\begin{equation*}
			[f_*\OO_V\xr{{\rm Trace}} \OO_Y\xr{} g_*\OO_Z]= \sum_{A} \ell_A \cdot [f_* \OO_V\xr{f_*(g_A^*)} g_*f_{A*}\OO_A
                                  \xr{g_*({\rm Trace})} g_*\OO_Z],
		\end{equation*}
		which is a straightforward computation.  
\end{proof}

\begin{proposition}\label{proposition-vanishing-implies-claim}
	If $\OO_Y\xr{g^*} Rg_*\OO_Z$ is an isomorphism in $D^b_c(Y)$, then $E(V\xr{} Y\xl{} Z)$ holds. 
	\begin{proof}
		For every non-empty open $U\subset Y$ the map 
		$$
		\Hom_{D^b_c(V)}(\OO_V,f^!\OO_Y) \xr{} \Hom_{D^b_c(f^{-1}(U))}(\OO_{f^{-1}(U)},f^!\OO_{U})
		$$
		is injective (Proposition \ref{proposition-vanishing-coh-of-support-dualizing-complex}). Hence, we may assume
         that $f$ is finite and flat and the statement follows from Lemma \ref{lem-finite-flat-case}. 
	\end{proof}
\end{proposition}

\begin{lemma}\label{lemma-composition}
	Consider a commutative diagram 
	$$
	\xymatrix
	{
		V\ar[d]
		&
		P_V \ar[l]\ar[d]
		\\
		Y
		&
		P\ar[l]
		&
		Z,\ar[l]
	}
	$$
	where $P\xr{} Y$ is projective and surjective, $P$ is regular, $P_V$ is integral, $P_V\xr{} V\times_Y P$ is a closed immersion, and $P_V$ is the only irreducible component of $V\times_Y P$ dominating $P$. 
	Assume $E(V\xr{} Y\xl{} P)$ and $E(P_V\xr{} P\xl{} Z)$ hold. Then $E(V\xr{} Y\xl{} Z)$ holds.    
	\begin{proof}
		Note that the irreducible components of $V\times_Y Z$  dominating $Z$ are
exactly the irreducible components of $P_V\times_P Z$ dominating $Z$. 
Thus the statement follows via a direct computation from Lemma \ref{lem-xi}, (1).
	\end{proof}
\end{lemma}

\begin{corollary}\label{cor-closed-implies-proj}
	Assume $E(V\xr{f} Y\xl{g} Z)$ holds for all closed immersions $g$. 
    Then $E(V\xr{f} Y\xl{g} Z)$ also holds for all projective morphisms $g$.
\end{corollary}

\begin{proof}
This follows directly from Proposition \ref{proposition-vanishing-implies-claim}, Lemma \ref{lemma-composition} 
and the equality $R\pi_*\sO_{\P^n_Y}=\sO_Y$, where $\pi: \P^n_Y\to Y$ is the projection.
\end{proof}

\begin{proposition}\label{proposition-divisor-implies-closed}
Assume $E(V\xr{f}Y \xl{g} Z)$ holds for any closed immersion $g$ of codimension 1.
Then $E(V\xr{f} Y\xl{g} Z)$ holds for any closed immersion $g$.
	\begin{proof}
Assume $g: Z\inj Y$ is a closed immersion. Since $Y$ and $Z$ are regular, $g$ is a regular closed immersion.
		Let $\tilde{Y}$ and $\tilde{V}$ denote the blow up of $Y$ in $Z$ and $V$ in $V\times_Y Z$, respectively. 
We form the commutative diagram 
		$$
		\xymatrix
		{
			V \ar[d]_{f}
			&
			\tilde{V}\ar[d]\ar[l]
			\\
			Y
			&
			\tilde{Y} \ar[l]^{\pi}
			&
			E, \ar[l]
		}
		$$
 where $E$ is the exceptional divisor.  Denote by $\pi_E: E\to Z$ the base change of $\pi$ along $g$.
As is well-known we have $R\pi_*\OO_{\tilde{Y}}=\OO_Y$.
Thus by Proposition \ref{proposition-vanishing-implies-claim} $E(V\xr{} Y\xl{} \tilde{Y})$ holds and by assumption
$E(\tilde{V}\xr{} \tilde{Y}\xl{}E)$ also holds. Hence $E(V\xr{f} Y\xl{g\circ\pi_E} E)$ holds 
by Lemma \ref{lemma-composition}.

As $\pi_E: E\to Z$ is a projective bundle we have $R\pi_{E*}\sO_E\cong \sO_Z$ and
the irreducible components of $V\times_Y Z$ correspond via $A\mapsto A\times_Z E$ to the irreducible components of
 $V\times_Y E$, further $\ell_A=\ell_{A\times_Z E}$. 
Set $E_A:=A\times_Z E$ and form the admissible squares
\[\xymatrix{V\ar[d]_f & A\ar[d]|-{f_A}\ar[l]_{g_A}\ar@{}[dl]|-{\sigma_A} & 
                                                           E_A\ar[l]_{\pi_{E_A}}\ar[d]^{f_{E_A}}\ar@{}[dl]|-{\sigma_{E_A}}\\
                Y   & Z\ar[l]^{g} & E\ar[l]^{\pi_E}.
}\]
We denote the  big outer admissible square by $\sigma_{V, E_A}$.
Let $A$ be an irreducible component of $V\times_Y Z$ dominating $Z$, 
Proposition \ref{proposition-vanishing-implies-claim} implies that $E(A\xr{f_{A}} Z \xl{\pi_E} E)$ holds, i.e.
		\begin{equation}\label{equation-cA-via-blowup}
	c_{f_A}=\left[\OO_A= R\pi_{E_A*}\sO_{E_A} \xr{R\pi_{E_A*}(c_{f_{E_A}})} 
                         R\pi_{E_A*}f^!_{E_A}\OO_E   \xr{\xi_{\sigma_{E_A}}} f_A^!R\pi_{E*}\sO_E=f_A^!\OO_Z\right].
		\end{equation}
We obtain
\begin{multline*}
			\left[ \OO_V \xr{} g_{A*}R\pi_{E_A*} \OO_{E_A} \xr{g_{A*}R\pi_{E_A*}(c_{f_{E_A}})}
               g_{A*}R\pi_{E_A*}f^!_{E_A}\OO_E \xr{\xi_{\sigma_{V,E_A}}} f^!g_*\OO_Z  \right] =\\
	\left[ \OO_V \xr{} g_{A*} R\pi_{E_A*} \OO_{E_A} \xr{} %\xr{g_{A*}R\pi_{E_A*}(c_{f_{E_A}})}
 g_{A*}R\pi_{E_A*}f^!_{E_A}\OO_E \xr{g_{A*}\xi_{\sigma_{E_A}}} g_{A*}f_A^!\OO_Z 
                  \xr{\xi_{\sigma_A}} f^!g_*\OO_Z  \right]\\
= \left[ \OO_V \xr{} g_{A*}\OO_{A} \xr{c_{f_A}} g_{A*}f^!_{A}\OO_Z \xr{\xi_{\sigma_A}} f^!g_*\OO_Z  \right].
		\end{multline*}
		Here, the first equality follows from Lemma \ref{lem-xi}, (1) and 
 the second equality follows from \eqref{equation-cA-via-blowup}. Thus $E(V\xr{f} Y\xl{g} Z)$ holds if and only if
$E(V\xr{f} Y\xl{g\circ \pi_E} E)$ holds, which proves the proposition.
\end{proof}
\end{proposition}

\section{Proofs}\label{proofs}

\begin{proof}[\bf Proof of Theorem \ref{thm2}.]
By Corollary \ref{cor-closed-implies-proj} and Proposition \ref{proposition-divisor-implies-closed}
we can assume that $g: Z\inj Y$ is a closed immersion of codimension 1.

{\em 1.Step: Reduction to $V$ being normal.}
 Let $\nu:\tilde{V}\xr{} V$ be the normalization. We claim that via the map
		$$
		\Hom(\OO_{\tilde{V}},(f\circ \nu)^!Rg_*\OO_Z)\xr{} 
                             \Hom(\OO_{V},f^!Rg_*\OO_Z), \quad a\mapsto  {\rm Tr}_{\nu}(\nu_*(a)\circ \nu^*),
		$$
		both sides of $E(\tilde{V}\xr{f\circ \nu} Y\xl{g} Z)$ are mapped to the corresponding side of
		$E(V\xr{f} Y\xl{g} Z)$. For the left-hand side this follows from the construction of $c_f$, 
             see \eqref{proposition-existence-fund-class1}. 
	       For the right-hand side first observe that if $B\xr{\nu_B} A$ is a finite surjective morphism between integral schemes, then 
              by Lemma \ref{lemma-push-forward}
		$$
		{\rm Tr}_{\nu_B}(\nu_{B*}(c_{f_A\circ \nu_B})\circ \nu_B^*)=\deg(B/A)\cdot c_{f_A}.
		$$
		Thus the claim follows from 
		$$
		\ell_A=\sum_B \ell_B\cdot \deg(B/A),
		$$	
		where $A$ is an irreducible component of $V\times_Y Z$, and the sum runs over all irreducible components 
          of $\tilde{V}\times_Y Z$ mapping to $A$, see \cite[Ex.~A.3.1]{F}.  

		{\em 2.Step: Reduction to $V$ being regular.} By our assumption on $g$, the following
              diagram is tor-independent 
		$$
		\xymatrix
		{
		V \ar[d]_{f}
		&
		V\times_Y Z \ar[l]_{g_1} \ar[d]^{f_1}\ar@{}[dl]|-{\sigma}
		\\
		Y
		&
		Z.\ar[l]^{g}
		}
		$$
Hence $\xi_{\sigma}(\OO_Z): g_{1*} f_1^!\OO_Z\xr{} f^!g_*\OO_Z$ is an isomorphism, by Lemma \ref{lem-xi}, (6).
Further $f_1$ satisfies condition \eqref{*}, by \cite[Prop. (5.6.5)]{EGAIV2}.
Since we have to prove an equality in $\Hom(\sO_V, f^!g_*\sO_Z)=\Hom(\sO_V, g_{1*}f_1^!\sO_Z)$
we can use Proposition \ref{proposition-vanishing-coh-of-support-dualizing-complex} 
to remove a codimension $\geq 2$ subset of $V$. Thus we can assume that $V$ is regular and 
the irreducible components of $V\times_Y Z$ are disjoint.

{\em 3.Step: End of proof.} 
Let us write $V\times_Y Z=\coprod_{i=1}^r A_i$ for the decomposition into connected components. 
Let $s\in \Hom(\OO_V,f^!g_*\OO_Z)$ be the element corresponding to 
$\OO_V\xr{c_f}f^!\OO_Y \xr{f^!(g^*)}f^!g_*\OO_Z$. We denote by $(s_{A_i})_i$ the image of $s$ via the map
		\begin{multline*}
\Hom(\OO_V,f^!g_*\OO_Z) \xr{=} \Gamma(V,\mathcal{H}^0(f^!g_*\OO_Z)) \xr{\cong} 
 \Gamma(V\times_Y Z,\mathcal{H}^0(f_1^!\OO_Z)) \\ 
\xr{=}\bigoplus_i \Gamma(A_i,\mathcal{H}^0(f_{A_i}^!\OO_Z)).
		\end{multline*}
		We claim that $s_{A_i}=0$ if $A_i$ does not dominate $Z$. Indeed, $V$ is regular, hence 
 $f_{A_i}:A_i\xr{} Z$ is a ci0 morphism. 
Using the definition of the fundamental class, Definition \ref{defn-FundamentalClass-ci0}, and
\eqref{delta2} one directly checks that $s_{A_i}=c_{f_{A_i}}$.
Thus Proposition \ref{prop-properties-FundCl-ci0}, (4) implies $s_{A_i}=0$.

Having no contributions from the non-dominant irreducible components, we may replace $Y$ by any open subset $U$ 
such that $U\cap Z\neq \emptyset$ and assume that $f^{-1}(U)\xr{} U$ is finite and flat. 
Now the statement follows from Lemma \ref{lem-finite-flat-case}. 
\end{proof}

\begin{proof}[\bf Proof: Theorem \ref{thm2}$\Rightarrow$ Theorem \ref{thm1}.]
We can assume that $Y$ is noetherian.  Since $f$ is birational, 
the only irreducible component of $X\times_Y X$ which dominates $X$ is 
the diagonal $\Delta$. Let $\sigma_\Delta$ be the commutative square 
 \[\xymatrix{ X\ar[d]_f & \Delta\ar[d]^{\simeq}\ar[l]_{\simeq}\ar@{}[dl]|-{\sigma_\Delta}\\
                   Y & X\ar[l]^f.
}\]
By Lemma \ref{lem-xi}, (3),  $\xi_{\sigma_\Delta}$ is the natural transformation $ \id\to f^!Rf_*$,
which by adjunction corresponds to the identity on $Rf_*$.  Theorem \ref{thm2} gives
$f^*\circ c_f=\xi_{\sigma_{\Delta}}$. Thus by adjunction the identity on $Rf_*\sO_X$
factors as $Rf_*\sO_X\to \sO_Y\xr{f^*} Rf_*\sO_X$. This proves Theorem \ref{thm1}.
\end{proof}

\begin{proof}[\bf Proof of Theorem \ref{thm-commutative-algebra}]
	Assertion (1) is equivalent to $H^i(X,\OO_X)=0$ for all $i>0$, by \cite[Theorem~4.1]{Lipman-CM}. Assertion (2) is equivalent to $H^i(X,\omega_X)=0$ for all $i>0$, in view of \cite{SanchoSalas}  (see \cite[Theorem~4.3]{Lipman-CM} and the following remark). Therefore it follows from Theorem \ref{thm1} by duality. 
\end{proof}

\begin{proof}[\bf Proof of Theorem \ref{thm-points}.]
Let $p$ be the characteristic of $k$. Again, we may assume that $Y$ is noetherian. Let $W_n\sO_X$ denote the sheaf of ($p$-typical) Witt vectors of length $n$ and
$W\sO_X=\varprojlim_n W_n\sO_X$ the sheaf of  ($p$-typical) Witt vectors. 
Set $W:=W(k)$ and $K_0:= {\rm Frac}(W)=W[\frac{1}{p}]$. By \cite[Cor. 1.3, Prop. 6.3]{BBE07}
it suffices to show
\[H^0(X_s, W\sO_{X_s})\otimes_W K_0= K_0\quad \text{and}\quad 
                            H^i(X_s, W\sO_{X_s})\otimes_W K_0=0\quad i\ge 1.\]
If $\kappa$ is the residue field of the image point of $s$ in $Y$, then the natural inclusion $W(\kappa)\inj W$
is \'etale. Thus it suffices to prove the above equalities in the case where  $s$ is a closed immersion, i.e.
$s\in Y$ is a closed point with residue field $k$.

Set $A=H^0(X_s,\sO_{X_s})$. Then $\Spec\, A\to \Spec\, k$ is finite, surjective and geometrically connected, hence radical. 
Since $k$ is perfect we obtain that $A$ is an artinian local $k$-algebra with residue field $k$. In particular
\[H^0(X_s, W\sO_{X_s})\otimes K_0=W(A)\otimes K_0= K_0,\]
where the second equality follows from $F \circ V=p=V\circ F$ on $W(A)$,
where $F: W(A)\to W(A)$, $(a_0,a_1,\ldots)\mapsto (a_0^p,a_1^p,\ldots)$ is the Frobenius morphism on the 
Witt vectors.

Denote by $f_p:X_p= X\times_\Z \F_p\to Y_p$ the base change of $f$ over $\F_p$. If $X_p=X$ then 
\begin{equation}\label{equation-Xp}
R^i f_{p*}\sO_{X_p}=0, \quad \text{for all $i\ge 1$,}
\end{equation}
follows immediately from Theorem \ref{thm1}. If $p\neq 0$ in $\OO_X$ then we can use the exact sequence
\[0\to\sO_X\xr{\cdot p}\sO_X\to \sO_{X_p}\to 0\]
to prove \eqref{equation-Xp}.

For all $n\ge 1$ we have an exact sequence of sheaves of abelian groups 
\[0\to W_{n-1}\sO_{X_p}\xr{V}W_n\sO_{X_p}\to \sO_{X_p}\to 0,\]
where $V$ is the Verschiebung, $V(a_0,\ldots, a_{n-2})=(0,a_0,\ldots, a_{n-2})$ and 
the map on the right is the restriction $(a_0,\ldots, a_{n-1})\mapsto a_0$.
Hence $R^i f_*W_n\sO_{X_p}=0$, for all $n,i\ge 1$. 
Further we have exact sequences for all $i\ge 1$
\[0\to R^1\varprojlim_n R^{i-1}f_* W_n\sO_{X_p}\to R^i f_* W\sO_{X_p}\to 
                  \varprojlim_n R^i f_*W_n\sO_{X_p}\to 0. \]
Thus also $R^i f_*W\sO_{X_p}=0$ for all $i\ge 1$. (For the case $i=1$ we use that the restriction maps
$f_*W_n\sO_{X_p}\to f_*W_{n-1}\sO_{X_p}$ are surjective, which implies the vanishing of 
$R^1\varprojlim_n f_* W_n\sO_{X_p}$.)

Now denote by $\sI$ the ideal sheaf of $X_s$ in $X_p$.  We obtain a long exact sequence
\[\cdots \to (R^i f_* W\sO_{X_p})\otimes K_0\to  (R^i f_{s*}W\sO_{X_s})\otimes K_0
                                                                              \to (R^{i+1} f_*W\sI)\otimes K_0\to \cdots .\]
By the above the term on the left vanishes and the term on the right vanishes by  \cite[Prop 4.6.1]{CR12},
which is a slight modification of \cite[Thm 2.4, (i)]{BBE07}. (In {\it loc. cit.} there is a general assumption that
the schemes considered have to be of finite type over a perfect field. One checks immediately 
that this assumption is not used in the parts we refer to.)
This proves Theorem \ref{thm-points}.
\end{proof}

\end{document}